\documentclass[12pt]{amsart}

\usepackage{amssymb}
\usepackage{amscd}

\title{Descendants of algebraic curves admitting two Galois points} 
\author{Satoru Fukasawa}

\address{Faculty of Science, Yamagata University, Kojirakawa-machi 1-4-12, Yamagata 990-8560, Japan} 
\email{s.fukasawa@sci.kj.yamagata-u.ac.jp} 

\thanks{The author was partially supported by JSPS KAKENHI Grant Number JP22K03223.}

\newtheorem{theorem}{Theorem}
\newtheorem{proposition}{Proposition}
 
\newtheorem{fact}{Fact}

\theoremstyle{definition}
\newtheorem{remark}{Remark}
\newtheorem{definition}{Definition} 
\newtheorem{problem}{Problem} 

\begin{document}
\begin{abstract} 
A connection between Galois points of an algebraic curve and those of a quotient curve is presented; in particular, the notion of a descendant of algebraic curves admitting two Galois points is introduced. 
It is shown that all descendants of a Fermat curve are Fermat curves; in particular, a Fermat curve does not have a descendant if and only if the degree is a prime.    
\end{abstract}

\maketitle 

\section{Introduction} 
The notion of a {\it Galois point} for a plane curve was introduced by Hisao Yoshihara in 1996. 
There are many interesting studies (see \cite{fukasawa1, miura-yoshihara, yoshihara, open}). 
The Galois group at a Galois point $P$ is denoted by $G_P$. 
In \cite{fukasawa2}, the author presented a criterion for the existence of a plane model with two Galois points of a smooth projective curve $X$, as follows.

\begin{fact} \label{criterion-outer} 
Let $G_1, G_2$ be different finite subgroups of ${\rm Aut}(X)$ and let $Q \in X$.
Then the following conditions are equivalent. 
\begin{itemize} 
\item[(I)] The following three conditions are satisfied: 
\begin{itemize}
\item[(a)] $X/{G_1} \cong \Bbb P^1$, $X/{G_2} \cong \Bbb P^1$,    
\item[(b)] $G_1 \cap G_2=\{1\}$, and
\item[(c)] $\sum_{\sigma \in G_1} \sigma (Q)=\sum_{\tau \in G_2} \tau (Q)$. 
\end{itemize}  
\item[(II)] There exists a birational embedding $\varphi: X \rightarrow \mathbb P^2$ of degree $|G_1|$ and different outer Galois points $P_1, P_2 \in \mathbb{P}^2 \setminus \varphi(X)$ exist for $\varphi(X)$ such that $G_{P_i}=G_i$ for $i=1, 2$ and points $\varphi(Q)$, $P_1$ and $P_2$ are collinear. 
\end{itemize} 
\end{fact} 

In \cite{fukasawa-higashine}, the author and Higashine generalized this theorem to $4$-tuples $(G_1, G_2, Q, H)$ with $H \vartriangleleft G_1$ and $H \vartriangleleft G_2$, and presented a necessary and sufficient condition for the existence of two Galois points of the quotient curve $X/H$. 
For the case where $X$ admits two Galois points, a sufficient condition for the existence of two Galois points of the quotient curve $X/H$ was given. 
As a subsequent study, this paper presents a connection between Galois points of $X$ and those of $X/H$; in particular, this paper introduces the notion of a descendant. 

For a finite subgroup $H$ of ${\rm Aut}(X)$, the quotient map is denoted by $f_H: X \rightarrow X/H$. 
For a point $Q \in X$, the image $f_H(Q)$ is denoted by $\overline{Q}$. 
Assume that $H$ is a normal subgroup of a finite subgroup $G \subset {\rm Aut}(X)$. 
It follows that $\sigma^*(k(X)^H)=k(X)^H$ for each $\sigma \in G$ with the pullback $\sigma^*: k(X) \rightarrow k(X)$.    
Therefore, there exists a natural homomorphism $G \rightarrow {\rm Aut}(X/H)$; $\sigma \mapsto \overline{\sigma}$, where $\overline{\sigma}$ corresponds to the restriction $\sigma^*|_{k(X)^H}$.  
The image is denoted by $\overline{G}$, which is isomorphic to $G/H$. 
As a special case of \cite[Theorem 2]{fukasawa-higashine} and a generalization of \cite[Corollary 2]{fukasawa-higashine}, the following holds. 

\begin{proposition} \label{main-outer} 
Let $H, G_1, G_2 \subset {\rm Aut}(X)$ be finite subgroups with $G_1 \ne G_2$ and let $Q \in X$. 
Assume that conditions (a), (b) and (c) of Fact \ref{criterion-outer} are satisfied for the triple $(G_1, G_2, Q)$, and that 
\begin{itemize}
\item[(d)] $G_i \not\subset H$, the set $G_i H$ is a group, and $H \vartriangleleft G_i H$ for $i=1, 2$. 
\end{itemize}
Then the following conditions are equivalent. 
\begin{itemize}
\item[(I)] The following two conditions are satisfied:  
\begin{itemize}
\item[(e)] $G_1H \cap G_2H=H$, and 
\item[(f)] $\sum_{\sigma \in G_1H} \sigma (Q)=\sum_{\tau \in G_2H} \tau (Q)$. 
\end{itemize}  
\item[(II)] There exists a birational embedding $\varphi: X/H \rightarrow \mathbb P^2$ of degree $|G_1H/H|$ and different outer Galois points $P_1, P_2 \in \mathbb{P}^2 \setminus \varphi(X/H)$ exist for $\varphi(X/H)$ such that $G_{P_i}=\overline{G_iH}$ for $i=1, 2$ and points $\varphi(\overline{Q})$, $P_1$ and $P_2$ are collinear. 
\end{itemize}
\end{proposition} 

With this proposition being used, the notions of a descendant, a maximality and a minimality are introduced, as follows.  

\begin{definition} 
Let $G_1, G_2 \subset {\rm Aut}(X)$ be finite subgroups with $G_1 \ne G_2$ and let $Q \in X$. 
Assume that conditions (a), (b) and (c) of Fact \ref{criterion-outer} are satisfied for the triple $(G_1, G_2, Q)$. 
\begin{itemize}
\item[(1)] If there exists a finite subgroup $H \subset {\rm Aut}(X)$ such that conditions (d), (e) and (f) of Proposition \ref{main-outer} are satisfied for the $4$-tuple $(G_1, G_2, Q, H)$, then $X/H$ (resp. $X$) is called a {\it descendant} of $X$ with respect to $(G_1, G_2, Q)$ (resp. an {\it ancestor} of $X/H$ with respect to $(\overline{G_1H}, \overline{G_2H}, \overline{Q})$). 
\item[(2)] $X$ is said to be {\it maximal} with respect to $(G_1, G_2, Q)$, if there does not exist an ancestor of $X$ with respect to $(G_1, G_2, Q)$.  
\item[(3)] $X$ is said to be {\it minimal} with respect to $(G_1, G_2, Q)$, if there does not exist a descendant with respect to $(G_1, G_2, Q)$.  
\end{itemize} 
\end{definition}

In Section 3, all descendants of several curves admitting two Galois points are determined. 

\section{Proof of Proposition \ref{main-outer}} 

The proof of Proposition \ref{main-outer} is similar to \cite[Theorem 2]{fukasawa-higashine}; however, we prove in detail for the convenience of the readers. 

\begin{proof}[Proof of Proposition \ref{main-outer}]
We consider the only-if part. 
Assume that conditions (a), (b) and (c) of Fact \ref{criterion-outer} and conditions (d), (e) and (f) of Proposition \ref{main-outer} are satisfied for the $4$-tuple $(G_1, G_2, Q, H)$. 
We would like to prove that conditions (a), (b) and (c) of Fact \ref{criterion-outer} are satisfied for the triple $(\overline{G_1H}, \overline{G_2H}, \overline{Q})$. 
Since $k(X/H)^{\overline{G_iH}}=k(X)^{G_iH} \subset k(X)^{G_i}$, by condition (a) and L\"{u}roth's theorem, the fixed field $k(X/H)^{\overline{G_iH}}$ is rational. 
It follows from condition (e) that $\overline{G_1H}\cap\overline{G_2H}=\{1\}$. 
Since
$$ \sum_{\sigma \in G_1H}\sigma(Q)=\sum_{H\sigma \in G_1H/H}\sum_{h \in H}h\sigma(Q), $$
it follows that
$$ (f_H)_*\left(\sum_{\sigma \in G_1H}\sigma(Q)\right)=\sum_{H\sigma \in G_1H/H}|H|\cdot\overline{\sigma(Q)}=|H|\sum_{\overline{\sigma} \in \overline{G_1H}}\overline{\sigma}(\overline{Q}). $$
It follows from condition (f) that 
$$|H|\left(\sum_{\overline{\sigma} \in \overline{G_1H}}\overline{\sigma}(\overline{Q})\right)=|H|\left(\sum_{\overline{\tau} \in \overline{G_2H}}\overline{\tau}(\overline{Q})\right). $$
Since $|H|\cdot D=0$ implies $D=0$ for any divisor $D$, we are able to cut the multiplier $|H|$. 
Condition (c) for the triple $(\overline{G_1H}, \overline{G_2H}, \overline{Q})$ is satisfied.

We consider the if part. 
By Fact \ref{criterion-outer}, we have that conditions (a), (b) and (c) of Fact \ref{criterion-outer} are satisfied for the triple $(\overline{G_1H}, \overline{G_2H}, \overline{Q})$.   
Since $\overline{G_1H} \cap \overline{G_2H}=\{1\}$, condition (e) is satisfied. 
By condition (c) for $X/H$, 
$$\sum_{\overline{\sigma} \in \overline{G_1H}}\overline{\sigma}(\overline{Q})=\sum_{\overline{\tau} \in \overline{G_2H}}\overline{\tau}(\overline{Q}). $$
Since $f_H^*(\overline{Q})=\sum_{h \in H}h(Q)$ for each $Q \in X$ (see, for example, \cite[III.7.1, III.7.2, III.8.2]{stichtenoth}), 
\begin{eqnarray*} 
f_H^*\left(\sum_{\overline{\sigma} \in \overline{G_1H}}\overline{\sigma}(\overline{Q})\right)&=&\sum_{\overline{\sigma} \in \overline{G_1H}}f_H^*(\overline{\sigma}(\overline{Q})) =\sum_{H\sigma \in G_1H/H}\sum_{h \in H}h\sigma(Q) \\
&=&\sum_{\sigma \in G_1H}\sigma(Q). 
\end{eqnarray*}
Similarly, 
$$ f_H^*\left(\sum_{\overline{\tau} \in \overline{G_2H}}\overline{\tau}(\overline{Q})\right)=\sum_{\tau \in G_2H}\tau(Q).$$
Condition (f) is satisfied. 
\end{proof} 

The following holds for Galois points which are smooth points. 

\begin{proposition} \label{main-inner} 
Let $H, G_1, G_2 \subset {\rm Aut}(X)$ be finite subgroups with $G_1 \ne G_2$ and let $P_1, P_2 \in X$. 
Assume that the following conditions are satisfied for the $5$-tuple $(G_1, G_2, H, P_1, P_2)$:  
\begin{itemize} 
\item[(a)] $X/{G_1} \cong \Bbb P^1$, $X/{G_2} \cong \Bbb P^1$,    
\item[(b)] $G_1 \cap G_2=\{1\}$, 
\item[(c)] $P_1+\sum_{\sigma \in G_1} \sigma (P_2)=P_2+\sum_{\tau \in G_2} \tau (P_1)$, and
\item[(d)] $G_i \not\subset H$, the set $G_i H$ is a group, and $H \vartriangleleft G_i H$ for $i=1, 2$. 
\end{itemize}
Then the following conditions are equivalent. 
\begin{itemize}
\item[(I)] The following three conditions are satisfied:  
\begin{itemize}
\item[(e)] $G_1H \cap G_2H=H$, 
\item[(f)] $\sum_{h \in H}h(P_1)+\sum_{\sigma \in G_1H} \sigma (P_2)=\sum_{h \in H}h(P_2)+\sum_{\tau \in G_2H} \tau (P_1)$, and 
\item[(g)] $H \cdot P_1 \ne H \cdot P_2$.  
\end{itemize}  
\item[(II)] There exists a birational embedding $\varphi: X/H \rightarrow \mathbb P^2$ of degree $|G_1H/H|+1$ such that points $\varphi(\overline{P_1}), \varphi(\overline{P_2}) \in \varphi(X/H)$ are different Galois points and $G_{\varphi(\overline{P_i})}=\overline{G_iH}$ for $i=1, 2$. 
\end{itemize}
\end{proposition} 

The notion of a descendant with respect to a $4$-tuple $(G_1, G_2, P_1, P_2)$ is defined similarly. 

\begin{remark}
Let $G_1, G_2, H$ be finite subgroups of ${\rm Aut}(X)$ with $G_1 \ne G_2$, and let $P_1, P_2 \in X$. 
Conditions (a), (b), (c), (d), (e), (f) and (g) of Proposition \ref{main-inner} are satisfied for the $5$-tuple $(G_1, G_2, H, P_1, P_2)$ with conditions (a), (b), (c), (d), (e) and (f) of \cite[Corollary 1]{fukasawa-higashine}.  
Therefore, curves admitting two Galois points obtained from \cite[Corollary 1]{fukasawa-higashine} are descendants. 
In particular, according to \cite[Section 5]{fukasawa-higashine}, the curve $x^3+y^4+1=0$ has a descendant $y^2+x^3+1=0$, namely, this curve is not minimal.  
\end{remark}

\section{Determination of all descendants of several curves admitting two Galois points} 
The characteristic of the base field is denoted by $p$.  
In characteristic $p=0$, all known examples of algebraic curves admitting two (outer) Galois points are the following (see \cite{fukasawa-speziali, miura-yoshihara, yoshihara}): 
\begin{itemize}
\item[(1)] Fermat curve $F_d$: $X^d+Y^d+Z^d=0$; 
\item[(2)] Takahashi's curve $T_m$: (the smooth model of) $X^{2m}+X^mZ^m+Y^{2m}=0$; 
\item[(3)] Another embedding of $T_m$.  
\end{itemize}

For an integer $l$, $\zeta_l$ denotes a primitive $l$-th root of unity. 
When we consider the Fermat curve $F_d$, we assume that $p=0$, or $d$ is not divisible by $p$ and $d-1$ is not a power of $p$. 
It is known that for the Fermat curve $F_d$, points $P_1=(1:0:0), P_2=(0:1:0) \in \mathbb{P}^2 \setminus F_d$ are Galois points (\cite{miura-yoshihara, yoshihara}). 
The Galois group at $P_i$ is denoted by $G_i^d$ for $i=1, 2$. 
Let $Q^d \in F_d \cap \{Z=0\}$. 
Then conditions (a), (b) and (c) of Fact \ref{criterion-outer} are satisfied for the triple $(G_1^d, G_2^d, Q^d)$. 
Let 
$$ K_l:=
\left\{ \left( \begin{array}{ccc}
\zeta_l^i & & \\
& \zeta_l^j & \\
& & 1
\end{array} 
\right) 
\ | \ 0 \le i, j \le l-1
\right\} \subset PGL(3, k) \cong {\rm Aut}(\mathbb{P}^2). $$
If $l$ divides $d$, then $K_l \subset G_1^d G_2^d$. 

\begin{theorem} 
Assume that $d \ge 4$, and that $d$ is not divisible by $p$ and $d-1$ is not a power of $p$. 
Let $H \subset {\rm Aut}(F_d)$ be a subgroup. 
Then $F_d/H$ is a descendant of $F_d$ with respect to $(G_1^d, G_2^d, Q^d)$ if and only if there exists an integer $l$ with $1< l < d$ such that $l$ divides $d$ and $H=K_l$. 
In this case, $F_d/H=F_n$, where $n=d/l$. 
In particular, the following hold. 
\begin{itemize} 
\item[(1)] $F_d$ is not maximal with respect to $(G_1^d, G_2^d, Q^d)$. 
\item[(2)] Any descendant of $F_d$ with respect to $(G_1^d, G_2^d, Q^d)$ is a Fermat curve. 
\item[(3)] $F_d$ is minimal with respect to $(G_1^d, G_2^d, Q^d)$ if and only if $d$ is a prime. 
\end{itemize}
\end{theorem}

\begin{proof}
We consider the ``if'' part. 
Let $d=n l$ with $l < d$. 
The set $G_1^d G_2^d$ coincides with the abelian group $K_d \cong \mathbb{Z}/d\mathbb{Z} \oplus \mathbb{Z}/d\mathbb{Z}$. 
We prove that conditions (d), (e) and (f) of Proposition \ref{main-outer} are satisfied for $(G_1^d, G_2^d, Q^d, K_l)$. 
Obviously, $G_i^d \not\subset K_l$ for $i=1, 2$. 
Note that 
\begin{eqnarray*}
G_1^d K_l &=&
\left\{ \left( \begin{array}{ccc}
\zeta_d^i & & \\
& \zeta_{l}^j & \\
& & 1
\end{array} 
\right) 
\ | \ 0 \le i \le d-1, \ 0 \le  j \le l-1
\right\}, \\ 
G_2^d K_l&=&
\left\{ \left( \begin{array}{ccc}
\zeta_l^i & & \\
& \zeta_d^j & \\
& & 1
\end{array} 
\right) 
\ | \ 0 \le i \le l-1, \ 0 \le  j \le d-1
\right\}, 
\end{eqnarray*} 
namely, the set $G_i^d K_l$ is an abelian group and $K_l \vartriangleleft G_i^d K_l$ for $i=1, 2$. 
Condition (d) is satisfied.  
It follows also that $G_i^d K_l \cap G_2^d K_l=K_l$. 
Condition (e) is satisfied. 
Since $G_1^d, G_2^d$ act on $F_d \cap \{Z=0\}$ transitively, it follows that 
$$ \sum_{\sigma \in G_1^d K_l} \sigma(Q^d)=\sum_{R \in F_d \cap \{Z=0\}} l R=\sum_{\tau \in G_2^d K_l}\tau(Q^d). $$
Condition (f) is satisfied. 

We consider the ``only if'' part. 
Let $S \subset {\rm Aut}(F_d)$ be a subgroup of order six generated by automorphisms 
$$ (X:Y:Z) \mapsto (Y: Z: X) \ \mbox{ and} \ (X:Y:Z) \mapsto (Y: X: Z). $$
It is known that ${\rm Aut}(F_d)=(G_1^d G_2^d)S \cong K_d \rtimes S$ (see, for example, \cite[Theorem 11.31]{hkt}). 
Assume that conditions (d), (e) and (f) are satisfied for $(G_1^d, G_2^d, Q^d, H)$. 
Assume that $H \not\subset K_d$. 
Let $h \in H \setminus (H \cap K_d)$. 
Then $h=g s$ for some $g \in K_d$ and $s \in S \setminus \{1\}$. 
Since $H \vartriangleleft G_1^d H$, it follows that $g_1 h g_1^{-1} \in H$ for any $g_1 \in G_1^d$. 
Since $K_d$ is abelian, it follows that 
$$h^{-1}(g_1 h g_1^{-1})=s^{-1} g^{-1}(g_1 g s g_1^{-1})=s^{-1} (g^{-1}g_1 g) s g_1^{-1}=(s^{-1}g_1 s) g_1^{-1} \in H, $$
namely, $s^{-1} g_1 s \in G_1^d H$. 
Similarly, for any $g_2 \in G_2$, $s^{-1} g_2 s \in G_2^d H$. 
It follows that $G_1^d H \supset K_d$ or $G_2^d H \supset K_d$. 
Assume that $G_1^d H \supset K_d$. 
Then $G_1^d H \supset K_d H \supset G_2^d H$, namely, $G_1^d H \cap G_2^d H=G_2^d H$.  
This is a contradiction to condition (e). 
We have $H \subset K_d$. 
Conditions $G_1^d G_2^d=K_d$ and $G_1^d H \cap G_2^d H=H$ imply that $H=(G_1^d\cap H)(G_2^d\cap H)$. 
Since $G_i^d H/H \cong G_i^d/(G_i^d \cap H)$ for $i=1, 2$, it follows that $|G_1^d \cap H|=|G_2^d \cap H|$, by condition (f).  
Let $l=|G_1^d \cap H|=|G_2^d \cap H|$. 
Then $1< l <d$, $l$ divides $d$, and $H=K_l$. 
\end{proof} 

\begin{remark}
Even if $d-1$ is a power of $p$, the following weak assertion can be proved according to the theorem \cite[Corollary 1.4]{fukasawa-speziali} of the author and Speziali: $F_d/H$ is a descendant of $F_d$ with respect to $(G_1^d, G_2^d, Q^d)$ if and only if there exists an integer $n$ with $1< n < d$ such that $n$ divides $d$ and $F_d/H=F_n$.
\end{remark}

When we consider Takahashi's curve $T_m$, we assume that $p=0$, or $2m$ is not divisible by $p$ and $2m-1$ is not a power of $p$. 
It is known that for Takahashi's curve $T_m$, points $P_1=(1:0:0), P_2=(0:1:0) \in \mathbb{P}^2 \setminus T_m$ are Galois points (\cite{fukasawa-speziali}). 
Let 
$$ \sigma(x, y)=(\zeta_m x, y), \ \tau(x, y)=\left(\frac{y^2}{x}, y\right).  $$
Then $G_{P_1} \cong D_{2m}$, and $G_{P_1}G_{P_2}=G_{P_1} \rtimes G_{P_2} \cong D_{2m} \rtimes \mathbb{Z}/(2m)\mathbb{Z}$, where $D_{2m}$ is the dihedral group of order $2m$. 
The Galois group at $P_i$ is denoted by $G_i^m$ for $i=1, 2$. 
Let $Q^m \in T_m \cap \{Z=0\}$. Then conditions (a), (b) and (c) of Fact \ref{criterion-outer} are satisfied for the triple $(G_1^m, G_2^m, Q^m)$. 
Let 
$$ K_l:=
\left\{ \left( \begin{array}{ccc}
\zeta_l^i & & \\
& \zeta_l^j & \\
& & 1
\end{array} 
\right) 
\ | \ 0 \le i, j \le l-1
\right\} \subset PGL(3, k). $$
If $l$ divides $m$, then $K_l \subset G_1^m G_2^m$. 
Condition (d), (e) and (f) are satisfied for $(G_1^m, G_2^m, Q^m, K_m)$, namely, a conic $X^2+X Z+Y^2=0$ is a descendant of $T_m$. 

\begin{theorem} \label{Takahashi's curve} 
Assume that $m \ge 3$, and that $2m$ is not divisible by $p$ and $2m-1$ is not a power of $p$ if $p >0$. 
Let $H \subset {\rm Aut}(T_m)$ be a subgroup.
Then $T_m/H$ is descendant of $T_m$ with respect to $(G_1^m, G_2^m, Q^m)$ if and only if there exists an integer $l$ with $1< l \le m$ such that $l$ divides $m$ and $H=K_l$. 
In this case, $T_m/H=T_n$, where $n=m/l$.   
In particular, the following hold. 
\begin{itemize} 
\item[(1)] $T_m$ is not maximal with respect to $(G_1^m, G_2^m, Q^m)$. 
\item[(2)] Any descendant of $T_m$ with respect to $(G_1^m, G_2^m, Q^m)$ is Takahashi's curve or a conic. 
\item[(3)] $T_m$ does not have a descendant other than a conic if and only if $m$ is a prime. 
\end{itemize}
\end{theorem}

\begin{proof} 
We consider the ``if'' part. 
Let $m=n l$ with $l \le m$. 
The set $G_1^{m}G_2^{m}$ is a group isomorphic to $D_{2m} \rtimes \mathbb{Z}/(2m)\mathbb{Z}$. 
The abelian group
$$ K_l=
\left\{ \left( \begin{array}{ccc}
\zeta_l^i & & \\
& \zeta_{l}^j & \\
& & 1
\end{array} 
\right) 
\ | \ 0 \le i, j \le l-1
\right\}. $$
is a normal subgroup of $G_1^m G_2^m$. 
We prove that conditions (d), (e) and (f) of Proposition \ref{main-outer} are satisfied for $(G_1^m, G_2^m, Q^m, K_l)$. 
Obviously, $G_i^m \not\subset K_l$ for $i=1, 2$. 
Since $G_1^m G_2^m=G_1^m \rtimes G_2^m$ and $K_l \vartriangleleft G_1^m G_2^m$, it follows that $G_i^m K_l$ is a group and $K_l \vartriangleleft G_i^m K_l$ for $i=1, 2$. 
Condition (d) is satisfied.  
Since $G_1^m K_l \cap G_2^m K_l$ is contained in an abelian group $G_2^m K_l$, it follows that $G_1^m K_l \cap G_2^m K_l=K_l$. 
Condition (e) is satisfied. 
Since $G_1^m, G_2^m$ act on $T_m \cap \{Z=0\}$ transitively, it follows that 
$$ \sum_{\sigma \in G_1^m K_l} \sigma(Q^m)=\sum_{R \in T_m \cap \{Z=0\}} l R=\sum_{\tau \in G_2^m K_l}\tau(Q^m). $$
Condition (f) is satisfied. 

We consider the ``only if'' part. 
Assume that there exists a subgroup $H \subset {\rm Aut}(T_m)$ such that conditions (d), (e) and (f) are satisfied for $(G_1^m, G_2^m, Q^m, H)$. 
It follows from a theorem of Kontogeorgis \cite{kontogeorgis} that ${\rm Aut}(T_m)=G_1^m G_2^m$. 
Conditions ${\rm Aut}(T_m)=G_1^m G_2^m$ and $G_1^m H \cap G_2^m H=H$ imply that $H=(G_1^m\cap H)(G_2^m\cap H)$. 
Since $G_1^m \cap H \vartriangleleft G_1^m$ and $G_1^m$ is a dihedral group, it follows that $G_1^m \cap H$ is contained in the normal subgroup $\langle \sigma \rangle$ of $G_1^m$ of index two.    
Since $G_i^m H/H \cong G_i^m/(G_i^m \cap H)$ for $i=1, 2$, it follows that $|G_1^m \cap H|=|G_2^m \cap H|$, by condition (f).  
Let $l=|G_1^m \cap H|=|G_2^m \cap H|$. 
Then $1< l \le m$, $l$ divides $m$, and $H=K_l$. 
\end{proof}

We consider another plane model of $T_m$. 
Let $\mu^{2m}=-1$, let 
$$ \sigma'(x, y)=(\zeta_m x, \zeta_m y), \ \tau'(x, y)=\left(\frac{y^2}{x}, \frac{y}{\mu^2} \right),  
$$ 
and let $G_3^m=\langle \sigma', \tau' \rangle$. 
Then $G_3^m$ is a cyclic group of order $2m$. 
Let $R^m$ be a point coming from $T_m \cap \{Y=0\}$. 
Then conditions (a), (b) and (c) of Fact \ref{criterion-outer} are satisfied for the triple $(G_1^m, G_3^m, R^m)$ (see \cite{fukasawa-speziali}). 
By the same method as in the proof of Theorem \ref{Takahashi's curve}, the following holds. 

\begin{theorem} 
Assume that $m \ge 3$, and that $2m$ is not divisible by $p$ and $2m-1$ is not a power of $p$ if $p >0$. 
Let $H \subset {\rm Aut}(T_m)$ be a subgroup. 
Then $T_m/H$ is a descendant of $T_m$ with respect to $(G_1^m, G_3^m, R^m)$ if and only if there exists an integer $l$ with $1< l \le m$ such that $l$ divides $m$ and $H=K_l$.   
In this case, $T_m/H=T_n$, where $n=m/l$. 
In particular, the following hold. 
\begin{itemize} 
\item[(1)] $T_m$ is not maximal with respect to $(G_1^m, G_3^m, R^m)$. 
\item[(2)] Any descendant of $T_m$ with respect to $(G_1^m, G_3^m, R^m)$ is Takahashi's curve or a conic. 
\item[(3)] $T_m$ does not have a descendant other than a conic if and only if $m$ is a prime. 
\end{itemize}
\end{theorem}

\begin{problem}
Give an example of a maximal curve. 
In particular, is the curve $y^3+x^4+1=0$ maximal?  
\end{problem}


\begin{thebibliography}{20} 
\bibitem{fukasawa1} S. Fukasawa, Galois points for a plane curve in arbitrary characteristic, Proceedings of the IV Iberoamerican Conference on Complex Geometry, Geom. Dedicata {\bf 139} (2009), 211--218. 
\bibitem{fukasawa2} S. Fukasawa, A birational embedding of an algebraic curve into a projective plane with two Galois points, J. Algebra {\bf 511} (2018), 95--101. 
\bibitem{fukasawa-higashine} S. Fukasawa and K. Higashine, A birational embedding with two Galois points for quotient curves, J. Pure Appl. Algebra {\bf 225} (2021), 106525, 10 pages. 
\bibitem{fukasawa-speziali} S. Fukasawa and P. Speziali, Plane curves possessing two outer Galois points, preprint, arXiv:1801.03198.  
\bibitem{hartshorne} R. Hartshorne, {\it Algebraic Geometry}, Graduate Texts in Mathematics {\bf 52}, Springer-Verlag, New York (1977). 
\bibitem{hkt} J. W. P. Hirschfeld, G. Korchm\'{a}ros and F. Torres, {\it Algebraic Curves over a Finite Field}, Princeton Univ. Press, Princeton (2008). 
%\bibitem{homma} M. Homma, Galois points for a Hermitian curve, Comm. Algebra {\bf 34} (2006), 4503--4511. 
%\bibitem{kty} M. Kanazawa, T. Takahashi and H. Yoshihara, The group generated by automorphisms belonging to Galois points of the quartic surface, Nihonkai Math. J. {\bf 12} (2001), 89--99. 
\bibitem{kontogeorgis} A. Kontogeorgis, The group of automorphisms of the function field of the curve $X^n+Y^m+1=0$, J. Number Theory {\bf 72} (1998), 110--136. 
\bibitem{miura-yoshihara} K. Miura and H. Yoshihara, Field theory for function fields of plane quartic curves, J. Algebra {\bf 226} (2000), 283--294. 
\bibitem{stichtenoth} H. Stichtenoth, {\it Algebraic Function Fields and Codes}, Universitext, Springer-Verlag, Berlin (1993). 
\bibitem{yoshihara} H. Yoshihara, Function field theory of plane curves by dual curves, J. Algebra {\bf 239} (2001), 340--355. 
\bibitem{open} H. Yoshihara and S. Fukasawa, List of problems, available at: \\ https://sites.google.com/sci.kj.yamagata-u.ac.jp/fukasawa-lab/open-questions-english
\end{thebibliography}
\end{document}